\documentclass[a4paper,12pt,reqno]{amsart}
 
\usepackage[margin=2cm]{geometry} 
\usepackage{amsmath,amsthm,amssymb,enumitem}
\usepackage{amsrefs, dsfont}
\usepackage[ocgcolorlinks,hyperfootnotes=false,colorlinks=true,citecolor=blue,linkcolor=blue,urlcolor=blue]{hyperref}

\newcommand{\fl}{f{\kern0.075em}l}
\newcommand{\norm}[1]{\left\lVert #1 \right\rVert}

\theoremstyle{plain}
\newtheorem{thm}{Theorem}[section]
\newtheorem{lem}[thm]{Lemma}
\newtheorem{prop}[thm]{Proposition}
\newtheorem{cor}[thm]{Corollary}

\theoremstyle{definition}
\newtheorem{defn}[thm]{Definition}
\newtheorem{exmp}[thm]{Example}
\newtheorem{Remark}[thm]{Remark}

\usepackage{mathtools}

\usepackage{orcidlink}

\setlist[enumerate]{itemsep=2mm, topsep=0mm}

\title[On the Bergman Metric near Exponentially Flat Infinite Type Boundary Points]
{On the Bergman Metric near Exponentially Flat Infinite Type Boundary Points}
\author{Ravi Shankar Jaiswal}
\address{Department of Mathematics, Southern University of Science and Technology, Xueyuan Avenue, Shenzhen, Guangdong, China 518055}
\email{ravi@sustech.edu.cn}
\date{\today}
\subjclass[2020]{Primary 32T27; Secondary 32A25, 32F45, 32Q20.}
\keywords{Bergman metric, Cheng--Yau conjecture, Infinite type boundary points.}

\begin{document}

\addtolength{\jot}{2mm}
\addtolength{\abovedisplayskip}{1mm}
\addtolength{\belowdisplayskip}{1mm}
\maketitle
\begin{abstract}
   We prove that the Bergman metric of a (possibly unbounded) pseudoconvex domain with an exponentially flat infinite type boundary point cannot be Einstein. Our result provides further evidence in support of the Cheng--Yau conjecture beyond the finite type setting.
\end{abstract}
\section{Introduction}
For a bounded pseudoconvex domain $D \subset \mathbb{C}^n$, $n\in\mathbb{N}$, there are two natural
biholomorphic invariant metrics defined on $D$. The first is the Bergman metric, which is a canonical K\"ahler metric. The second is the complete K\"ahler--Einstein metric, whose existence and uniqueness (upto a positive scaling) was established by Cheng--Yau \cite{Cheng Yau} and Mok--Yau \cite{Mok Yau}. These two metrics capture the interplay between function theory and complex geometry of the domain.

The existence of these two canonical metrics naturally raises the question of
when they coincide. Motivated by this problem, Yau \cite{Yau} conjectured that
the Bergman metric of a bounded pseudoconvex domain is complete and Einstein if
and only if the domain is biholomorphic to a bounded homogeneous domain. Since
every $C^2$-smooth bounded homogeneous domain is biholomorphic to the unit ball by
the theorem of Wong \cite{Wong} and Rosay \cite{Rosay}, and the Bergman metric of every $C^1$-smooth bounded pseudoconvex domain is complete by the theorem of Ohsawa \cite{Oh81}, Yau's conjecture
extends an earlier conjecture of Cheng \cite{Cheng79}, which states that the
Bergman metric of a $C^{\infty}$-smooth bounded strongly pseudoconvex domain is
K\"ahler--Einstein if and only if the domain is biholomorphic to the unit ball.
Combining the conjectures of Cheng and Yau yields the celebrated Cheng--Yau conjecture.

\textbf{Conjecture (Cheng--Yau {\cites{Cheng79,Yau}}).}
A $C^{\infty}$-smooth bounded pseudoconvex domain in \(\mathbb{C}^n\) is
\emph{Bergman--Einstein}, i.e., its Bergman metric is K\"ahler--Einstein if
and only if it is biholomorphic to the unit ball of the same dimension.

The conjecture has attracted considerable attention in recent years. Cheng's
conjecture for $C^{\infty}$-smooth bounded strongly pseudoconvex domains was first
established in complex dimension two by Fu--Wong \cite{Fu_Wong} and
independently by Nemirovski--Shafikov \cite{Nemiroski_Shafikov}, and was later
resolved in all dimensions by Huang--Xiao \cite{Huang_Xiao}. More recently,
Savale--Xiao \cite{Savale_2025} proved the conjecture for $C^{\infty}$-smooth bounded
finite type pseudoconvex domains in \(\mathbb{C}^2\). Earlier, Fu--Wong
\cite{Fu_Wong} had obtained the same conclusion for $C^{\infty}$-smooth bounded complete
Reinhardt pseudoconvex domains of finite type in \(\mathbb{C}^2\). Very
recently, Hsiao--Huang--Li \cite{HHL26} established the Cheng--Yau
conjecture for bounded real analytic pseudoconvex domains and for $C^{\infty}$-smooth
bounded convex domains of finite type in the sense of D'Angelo
\cite{D'Angelo1982}. Subsequent generalizations were obtained, including Stein manifolds and Stein
spaces with compact strongly pseudoconvex boundaries; see Huang--Li \cite{Huang_Li},
Ebenfelt--Xiao--Xu \cites{EXX22, EXX24}, Ganguly--Sinha \cite{GS26} and references therein. 
Related variations of Cheng’s conjecture were also studied by S. Li in his works  \cites{Li05, Li09, Li16} and in a recent paper by Yuan \cite{Yuan25}.

The Cheng--Yau conjecture has also been investigated in the setting of
unbounded pseudoconvex domains. Huang--James--Li \cite{HJL25} proved that
the Bergman metric of a (possibly unbounded) pseudoconvex domain in
\(\mathbb{C}^{n + 1}\), cannot be Einstein if its boundary contains a
non-smooth strongly pseudoconvex polyhedral boundary point. More recently,
Hsiao--Huang--Li \cite{HHL26} showed that the Bergman metric of a (possibly
unbounded) pseudoconvex domain in \(\mathbb{C}^{n + 1}\), cannot be Einstein
if its boundary contains a $C^{\infty}$-smooth non-strongly pseudoconvex \(h\)-extendible
boundary point. 

In this paper, we continue this line of investigation for a class of infinite type domains. We prove that the Bergman metric cannot be Einstein by establishing the asymptotic behavior of the Bergman biholomorphic invariants. This strategy for proving that the Bergman metric cannot be Einstein was previously employed by Huang--James--Li \cite{HJL25} and Hsiao--Huang--Li \cite{HHL26}.

We first state the non-tangential asymptotic behavior of the Bergman kernel, the Bergman metric, the Bergman canonical invariant, and the Ricci curvature at an exponentially flat infinite type boundary point, see Definition \ref{exp flat fun}, in the following theorem.
%Our approach closely follows the methods developed by Huang--James--Li \cite{HJL25}, suitably adapted to the setting of exponentially flat infinite type boundary points.
\begin{thm}\label{1.3}
    Let $D \subset \mathbb{C}^{n + 1}$ be a pseudoconvex domain, possibly unbounded, and let $0 \in bD$. Assume $0$ is an exponentially flat boundary point and $q(t) = (-t, 0)$, then, for $\xi \in \mathbb{C}^{n + 1} \setminus \{0\}$, 
    \begin{enumerate}[labelsep=10mm,leftmargin=20mm,itemsep=5mm]
    \item $\begin{aligned}[t]
        \lim_{t \to 0^+}\frac{\kappa_D(q(t))}{t^{-2}d^*(t)^{-2n}} = \frac{1}{4\pi\operatorname{vol}(B_n(0,1))},
    \end{aligned}$
        \item $\begin{aligned}[t]
     \lim_{t \to 0^+} \frac{B_D(q(t); \xi)}{\sqrt{{|\xi_1|^2}/{2t^2} + (n+1) {|\xi'|^2}/{d^*(t)^2}}} = 1, 
    \end{aligned}$
        \item $\begin{aligned}[t]
        \lim_{t \to 0^{+}} J_D(q(t)) = \frac{2 \pi^{n + 1}(n + 1)^n}{n!}, \text{ and}
    \end{aligned}$
    \item $\begin{aligned}[t]
    \lim_{t \to 0^{+}} R_D(q(t); \xi) = -1.
    \end{aligned}$
    \end{enumerate} 
\end{thm}
The proof of Theorem \ref{1.3} is obtained by combining localization and scaling techniques. The key ingredient is a localization of these biholomorphic invariants on possibly unbounded pseudoconvex domains with exponentially flat infinite type boundary points. The proof relies on the construction of suitable plurisubharmonic functions and using them as weights in H\"ormander's $L^2$-estimates for the $\overline{\partial}$-problem to localize the associated extremal functions. This method was previously employed by Nikolov \cite{Nikolai}, James \cite{James} and Huang--James--Li \cite{HJL25} to establish localization of extremal functions associated to the Bergman space on possibly unbounded pseudoconvex domains with local peak points. Localization Lemma \ref{localisation lemma 4} extends the localization results established in \cites{Ravi_Bulletin, Ravi_Synergies} for bounded pseudoconvex domains. Combined with the scaling method developed in \cite{Ravi_Bulletin}, this localization yields the desired non-tangential boundary asymptotics. It extends the asymptotic theory of the Bergman kernel, the Bergman metric, the Bergman canonical invariant, and the Ricci curvature from bounded, see \cites{Ravi_Bulletin, Ravi_Synergies},  to possibly unbounded pseudoconvex domains with exponentially flat infinite type boundary points.

We now state one of the main results of this article.
%prove that the Bergman metric of a (possibly unbounded) pseudoconvex domain containing an exponentially flat infinite type boundary point cannot be Einstein using the asymptotic behavior of the Bergman biholomorphic invariants established in the preceding theorem. This strategy for proving that the Bergman metric cannot be Einstein was previously employed by Huang--James--Li \cite{HJL25} and Hsiao--Huang--Li \cite{HHL26}.
\begin{thm}\label{Main theorem}
    Let $D \subset \mathbb{C}^{n + 1}$ be a pseudoconvex domain, possibly unbounded, and let $0 \in bD$. Assume $0$ is an exponentially flat boundary point, then the Bergman metric of $D$ is well defined on a nonempty open subset of $D$, denoted by $D^{*}$, and it cannot be Einstein on $D^{*}$.
\end{thm}
\begin{cor}
    Let $D \subset \mathbb{C}^{n + 1}$ be a bounded pseudoconvex domain. Assume $0$ is an exponentially flat boundary point, then the Bergman metric can not be Einstein.
\end{cor}
%The proof of Theorem \ref{Main theorem} is based on the non-tangential asymptotic behavior of the Bergman kernel, the Bergman metric, the Bergman canonical invariant, and the Ricci curvature at exponentially flat infinite type boundary points.
%, and provide further evidence in support of the Cheng--Yau conjecture beyond the finite type setting.

The article is organized as follows. In Section~\ref{Section 2}, we recall the necessary preliminaries and several known results. In Section~\ref{section 3}, we establish the asymptotic behavior of the Bergman kernel, the Bergman metric, the Bergman canonical invariant, and the Ricci curvature at exponentially flat infinite type boundary points. Finally, in Section~\ref{Section 4}, we prove Theorem~\ref{Main theorem}.

\textbf{Notations.} We denote the unit disc in $\mathbb{C}$ by $\mathbb{D}$, the unit ball in $\mathbb{C}^{n + 1}$ by $\mathbb{B}^{n + 1}$, $n \in \mathbb{N}$, and the Euclidean ball in $\mathbb{C}^n$ with center $p$ and radius $r$ by $\mathbb{B}^n(p,r)$.

\section{Preliminaries}\label{Section 2}
Let $D \subset \mathbb{C}^{n}$ be a domain. The Bergman space of $D$, denoted by $A^2(D)$, consists of holomorphic functions on $D$ that are square integrable with respect to the Lebesgue measure. 
We assume that $A^2(D) \neq \{0\}$. Since $A^2(D)$ forms a closed subspace of $L^2(D)$, it is a non-trivial Hilbert space. Let $\{\phi_j\}_{j = 1}^{N}$ be an orthonormal basis for $A^2(D)$ with respect to the standard inner-product, where $N$ can be finite or infinite. The Bergman kernel function of $D$ on the diagonal is then defined by
\[\kappa_D(z) = \sum_{j = 1}^{N} \phi_j(z)\overline{\phi_j(z)}, \quad \text{for } z \in D.\]
Additionally, the Bergman metric of $D$, when well defined, is given by
\begin{align}
    B_{D} = \sum_{i, j = 1}^{n} g_{i\bar{j}} dz_i \otimes d\bar{z}_{j}, \quad \text{where} \quad g_{i \bar{j}} = \frac{\partial^2 \operatorname{log}\kappa_{D}}{\partial z_i \partial \bar{z}_j},
\end{align}
and the Bergman norm, when well defined, is defined as
\begin{align} 
B_D(z;\xi)=\sqrt{\sum_{i,j=1}^{n}g_{i\bar{j}}(z)\xi_i\overline{\xi}_j},
\end{align}
for $z \in D$ and $\xi \in \mathbb{C}^n$.

The Bergman metric on a domain \(D\subset\mathbb{C}^n\) may not be defined at every point if \(D\) is unbounded.
It is well defined at $z \in D$ if and only if $A^2(D)$ is \emph{base-point free} at $z$ (i.e., there exists $f \in A^2(D)$ such that $f(z) \neq 0$) and \emph{separate holomorphic directions} at $z$ (i.e., for every nonzero holomorphic tangent vector $X \in T_z^{(1, 0)}D$, there exists $f \in A^2(D)$ such that $f(z) = 0$ and $X(f)(z) \neq 0$); see \cite{Kobayashi}.

We denote
\[
D^{*}
=
\left\{
z\in D:\;
A^2(D)\text{ is base-point free and separates holomorphic directions at }z
\right\}.
\]
Throughout this article, we assume that \(D^{*}\neq\emptyset\).
\begin{Remark} If $D$ is a bounded domain, then it is easy to see that $D^{*} = D$. More generally,
if \(D\subset\mathbb{C}^n\) is a pseudoconvex domain, and assume \(p\in bD\) is a local holomorphic peak point, then, by \cite{Nikolai 2002}*{Theorem $2$}, $D^{*}$ contains a neighbourhood of $p$ intersected with $D$.
\end{Remark}

The Bergman canonical invariant is given by
\begin{align}\label{Bergman invariant}
    J_D(z) = \frac{\operatorname{det}G_D(z)}{\kappa_D(z)} \quad \text{for} \,\, z \in D^{*},
\end{align}
where $G_D(z) = [g_{j \bar{k}}(z)]$.

The Ricci curvature tensor of the Bergman metric $B_D$ is given by 
\begin{align}
    R_{D} = \sum_{i, j = 1}^{n} R_{i \bar{j}}dz_i \otimes d\bar{z}_j, \quad \text{where} \quad R_{i \bar{j}} = -\frac{\partial^2 \operatorname{log} \operatorname{det}G_{D}}{\partial z_i \partial \bar{z}_j},
\end{align}
and the Ricci curvature of the Bergman metric along the direction $\xi \in \mathbb{C}^n \setminus \{0\}$ is given by
\begin{align}
     R_{D}(z; \xi) = \frac{\sum_{i, j = 1}^{n} R_{i \bar{j}}(z)\xi_i \bar{\xi}_j}{B_{D}(z;\xi)^2} \quad \text{for} \, \, z \in D^{*}.
\end{align}

The Bergman metric \(B_{D}\) is a K\"ahler metric on \(D^{*}\). We say that \(B_{D}\) is Einstein on \(D^{*}\) if there exists a constant $c$ such that
\[
R_{D}=c\,B_{D}.
\]
We now introduce the following extremal functions, which play a crucial role in localizing and estimating $\kappa_D, B_D, J_D$ and $R_D$. For $z \in D$ and $\xi \in \mathbb{C}^n$,
\begin{align}
     I_D^0(z) &= \operatorname{sup}\left\{|f(z)|^2: \norm{f}_{A^2(D)} \leq 1\right\},\\
     I_D^1(z; \xi) &= \operatorname{sup}\left\{|f'(z)\xi|^2: \norm{f}_{A^2(D)} \leq 1, f(z) = 0\right\},\\
    \lambda_{D}^k(z) &= \operatorname{sup}\left\{\left|\frac{\partial f}{\partial z_k}(z)\right|^2: \norm{f}_{A^2(D)} \leq 1, \, f(z) = \frac{\partial f}{\partial z_j}(z) = 0, \text{ for } 1 \leq j < k\right\},\\
    I_{D}(z; \xi) &= \operatorname{sup}\left\{\xi f''(z)\overline{G}_{D}^{-1}(z)\overline{f''(z)}\xi^{*} :
    \norm{f}_{A^2(D)} \leq 1, \, f(z) = 0, \, f'(z) = 0\right\}, \text{ and}\\
    M_{D}(z; \xi) &= \operatorname{sup}\left\{\kappa_{D}^{n -1}(z) \xi f''(z)\, \overline{\operatorname{{ad}}G_{D}(z)} \, \overline{f''(z)} \xi^{*}: \norm{f}_{A^2(D)} \leq 1, f(z) = 0, f'(z) = 0\right\},
\end{align}
where $\xi^{*}$ is the tanspose conjugate vector of $\xi$, $f'' = \left(\frac{\partial^2 f}{\partial z_i \partial z_j}\right)_{n \times n}$, and $\operatorname{ad}G_D$ is the adjoint matrix whose $(ij)-$th entry is the cofactor of $g_{j \bar{i}}$.

The following proposition, known as the Bergman--Fuks formula, gives a useful representation of the Bergman kernel
and the Bergman metric in terms of the above extremal functions.
\begin{prop}[\cite{Jarnicki}*{Theorem 12.7.5}]\label{Fuchs}
    Let $D \subset \mathbb{C}^{n}$ be a domain,  then, for $z \in D^{*}$
    and $\xi \in \mathbb{C}^{n}$, we have
\begin{align}
        \kappa_D(z) &= {I_D^0(z)}, \quad \text{and} \quad
        B_D^2(z; \xi) = \frac{I_D^1(z; \xi)}{I_D^0(z)}.  
\end{align}
\end{prop}
The following proposition gives a useful representation of the Bergman canonical invariant and Ricci curvature in terms of extremal functions.
\begin{prop}[\cite{Krantz-Yu 1996}*{Proposition $2.1$}, \cite{James}, \cite{HJL25}]\label{Pro 2.1}
Let $D$ be a domain in $\mathbb{C}^n$, then, for $z \in D^{*}$ and $\xi \in \mathbb{C}^n\setminus\{0\}$, we have
\begin{enumerate}
    \item $J_{D}(z) = \frac{\lambda_{D}(z)}{\kappa_{D}^{n + 1}(z)}$, where $\lambda_{D}(z) = \lambda_{D}^1(z) \dots \lambda_{D}^n(z)$.
    \item $R_{D}(z; \xi) = (n + 1) - \frac{1}{B_{D}^2(z; \xi)\kappa_{D}(z)}I_{D}(z; \xi)$.
    \item $R_{D}(z; \xi) = (n + 1) - \frac{1}{B_{D}^2(z; \xi)\kappa_{D}^{n + 1}(z) J_{D}(z)}M_{D}(z; \xi)$.
\end{enumerate}
\end{prop}
We will utilize the following proposition to prove the localization of the extremal function $I_D$ in Lemma \ref{localisation lemma 4}.
\begin{prop}[\cite{Krantz-Yu 1996}*{Propositon $2.2$}, \cite{James}]\label{est. of G} Let $D_1 \subset D_2$ be two domain in $\mathbb{C}^n$, then, for $z \in D_2^{*}$ and $\xi \in \mathbb{C}^n$, we have
\begin{enumerate}
[labelsep=10mm,leftmargin=20mm,itemsep=5mm]
    \item $\begin{aligned}[t]
    \xi \overline{G}_{D_2}^{-1}(z)\xi^{*} \geq \frac{\kappa_{D_2}(z)}{\kappa_{D_1}(z)} \xi \overline{G}_{D_1}^{-1}(z)\xi^{*}.
    \end{aligned}$
\item $\begin{aligned}[t]
     M_{D_2}(z; \xi) \leq M_{D_1}(z; \xi).
     \end{aligned}$
\end{enumerate}
\end{prop}
We will use the following lemma to obtain the exact asymptotic behavior of biholomorphic invariants (see Theorem \ref{1.3}).
\begin{lem}[\cites{Ravi_Bulletin, Ravi_Synergies}]\label{Formula for kernel}
The Bergman kernel, the Bergman metric, the Bergman canonical invariant, and the Ricci curvature of $\mathbb{D} \times \mathbb{B}^{n}$ at $0$ are as follows.
\begin{align*}
\kappa_{\mathbb{D} \times \mathbb{B}^{n}}(0) &= \frac{n!}{\pi^{n + 1}}, \qquad \qquad \qquad \quad
B_{\mathbb{D} \times \mathbb{B}^{n}}(0; \xi) = \sqrt{2|\xi_1|^2 + (n + 1)|\xi'|^2},\\
J_{\mathbb{D} \times \mathbb{B}^{n}}(0) &= \frac{2\pi^{n + 1}(n + 1)^n}{n!}, \qquad \quad 
R_{\mathbb{D} \times \mathbb{B}^{n}}(0; \xi) = -1,
\end{align*}
for $\xi = (\xi_1, \xi') \in \mathbb{C} \times \mathbb{C}^{n} \setminus \{0\}$.
\end{lem}
We next recall the definition of an \emph{exponentially flat} boundary point for a domain \(D \subset\mathbb{C}^{n + 1}\), as introduced in \cite{Ravi_Bulletin}.
\begin{defn}\label{exp flat fun}
    Let $D \subset \mathbb{C}^{n + 1}$ be a domain and let $0 \in bD$. The boundary point
$0$ is said to be \emph{exponentially flat} if there exists a local defining function of $D$ near the origin of the form
\begin{align}\label{1}
    \rho(z) = \operatorname{Re}z_1 + \phi \left(|z'|^2\right),
\end{align}
 where $\phi : \mathbb{R} \to \mathbb{R}$ is a $C^{\infty}$-smooth function that is \emph{exponentially flat at the origin}, defined below.
\end{defn}
 \begin{defn}
    A $C^{\infty}$-smooth function $\phi: \mathbb{R} \to \mathbb{R}$ is said to be
    \emph{exponentially flat at the origin}, if it satisfies the following properties:
    \begin{enumerate}
    \item $\phi(x) = 0,$ for $x \leq 0$, 
    \item there exists $\epsilon_0 > 0$ such that $\phi''(x) > 0$, for $0 < x < \epsilon_0$, and
    \item the function 
    \begin{align*}
        \psi(x) := \begin{cases} 
      -1/ \operatorname{log}(\phi(x)), &  \text{if } 0 < x < \epsilon_0,\\
      0, & \text{if } x = 0,
    \end{cases}
    \end{align*}
    satisfies
    \begin{align}\label{asy of psi}
        \lim_{x \to 0^{+}}\frac{\psi(x)}{x^m} = C, \text{\,for some $m, C > 0$.}
    \end{align}
\end{enumerate}
\end{defn}
\begin{Remark}
If $\phi$ is exponentially flat at the origin, then,
\begin{align}\label{5}
    \phi^{(k)}(0) = 0, \quad \text{for all } k \in \mathbb{N}.
\end{align}
Using \eqref{5}, we can easily see that exponentially flat boundary points are points of infinite type in the sense of D'Angelo \cite{D'Angelo 1982}.
\end{Remark}
\begin{exmp}
For $m \in \mathbb{R}^{+}$,
\begin{align*}
    \phi(x) = \begin{cases} 
      0,  &\text{if } x \leq 0,\\
      \operatorname{exp}\left(-1/x^m\right),  & \text{if } x > 0,
      \end{cases}
\end{align*}
is exponentially flat at the origin.
\end{exmp}
We prove localization of extremal functions (Lemma \ref{localisation lemma 4}) on a (possibly unbounded) pseudoconvex domain with an exponentially flat boundary point by utilising the following version of the solution to the $\overline{\partial}$-problem due to H\"ormander \cite{Hormander 1965}. This version of the solution to the $\overline{\partial}$-problem is also used in Nikolov \cite{Nikolai}, and Huang--James--Li \cite{HJL25} to localize the extremal functions.
\begin{thm}[\cite{Hormander 1965}*{Theorem $2.2.1'$}, \cite{Gallagher}*{Theorem $5$}]\label{Prop L2}
Let $D \subset \mathbb{C}^n$ be a pseudoconvex domain, possibly unbounded, and let
$\varphi: D\to[-\infty,\infty)$ 
be a plurisubharmonic function. Assume that:

\begin{enumerate}
\item $U \subset D$ is open and \(\varphi(z)-c|z|^2\) is plurisubharmonic on \(U\) for some constant \(c>0\),

\item \(v\in L^2_{(0,1)}(D,\varphi)\) is a \(C^\infty\)-smooth \((0,1)\)-form on $D$ satisfying $\bar{\partial}v=0$
and $\operatorname{supp}(v)\subset U$.
\end{enumerate}

Then, there exists a \(C^\infty\)-smooth function \(u\) on $D$ such that
\[
\bar{\partial}u=v
\]
and
\begin{align}\label{16}
\int_{D}|u|^2 e^{-\varphi}\, dV
\le
\frac{1}{c}
\int_{D}|v|^2 e^{-\varphi}\, dV.
\end{align}
\end{thm}
\section{Asymptotic behavior at an exponentially flat boundary point of a pseudoconvex unbounded domain}\label{section 3}
In this section, we first prove the localization Lemma \ref{localisation lemma 4} for extremal functions. Then, combining this with the scaling Lemma \ref{ScalingLemma2} from \cite{Ravi_Bulletin} to obtain the nontangential asymptotic behavior of the Bergman kernel, the Bergman metric, the Bergman canonical invariant, and the Ricci curvature at exponentially flat infinite type boundary points of possibly unbounded pseudoconvex domains in $\mathbb{C}^{n + 1}$.

Let $D \subset \mathbb{C}^{n + 1}$ be a domain with $0 \in bD$ and let $q(t) = (-t, 0)$. If $0$ is an exponentially flat boundary point, then there exists a neighbourhood $U$ of the origin, such that
\begin{align}\label{*}
    D \cap U = \{z \in U: \rho(z) < 0\},
\end{align}
where the defining function $\rho$ is defined as in \eqref{1}, and $D \cap U$ is convex.

Since 
\begin{align}\label{peak fun.}
    h(z)=\operatorname{exp}(-(-z_1)^{\alpha}), \quad \text{for } \alpha \in (0, 1),
\end{align}
is a local holomorphic peak function of $D$ at $0$, there exist a bounded
plurisubharmonic function \(\theta\) on \(D\) and a constant \(r\in(0,1)\)
such that
\begin{align}\label{40}
    \theta(z)=|z|^2,
\end{align}
for $z\in D\cap\mathbb{B}^{n+1}(0,r)$ (see, e.g., \cite{Nikolai}*{p. $356$} and \cite{Jaiswal_Kar_arXiv}*{Lemma $3.3$}).

Let $W \subset \subset \mathbb{B}^{n + 1}(0, r/2) \cap U$ be a convex neighbourhood of the origin and assume 
\[\Omega = \{z \in \mathbb{C}^n: \rho(z) < 0\},\]
then $D \cap W = \Omega \cap W$.

To prove Theorem \ref{1.3}, we follow the construction of the domain $D$ as given in \cites{Ravi_Bulletin, Ravi_Synergies, Jaiswal_Kar_arXiv}.

We first cut out a portion of $D$ near the origin, depending on the point $q(t)$, and denote it by $D_t^{\epsilon}$ for $\epsilon > 0$. We then localize extremal functions from $D$ to $D_t^{\epsilon}$. Next, we apply the scaling map $\Sigma$ and biholomorphic map $f$ on $D_t^{\epsilon}$ so that the resulting domain converges to $\mathbb{D} \times \mathbb{B}^n$, we refer to as the scaling lemma. Using the scaling lemma, we conclude that extremal functions of $f \circ \Sigma(D_t^{\epsilon})$ converge to extremal functions of $\mathbb{D} \times \mathbb{B}^n$. Combining these arguments, we obtain the asymptotic behavior of extremal functions.

For small $\epsilon, t > 0$, define
\begin{align}
    D_t^{\epsilon} &= \left\{z \in W : \rho(z) < 0, \operatorname{Re}z_1 > - t^{{1}/{(1+\epsilon)^2}}\right\},
\end{align}
and $h_\epsilon : D_t^{\epsilon} \to \mathbb{D}$ is the holomorphic peak function of $D_t^{\epsilon}$ at zero, given by 
\begin{align}\label{hepsilon}
    h_\epsilon(z) = \operatorname{exp}\left(-(-z_1)^{{1}/{(1+ \epsilon)}}\right). 
\end{align}
Let \( d^{*}(t) \), \( d_{1}^{\epsilon}(t) \), and \( d_{2}^{\epsilon}(t) \) denote the complex tangential distances from \( -t e_1 \), \( -t^{\frac{1}{1+\epsilon}} e_1 \), and \( -t^{\frac{1}{(1+\epsilon)^2}} e_1 \), respectively, to \( b\Omega \), i.e.,
\begin{align*}
    d^*(t) &= 
    \operatorname{sup}\{|z'|: z \in W, \phi(|z'|^2) \leq t\} = \sqrt{\phi^{-1}(t)}, \\
    d_1^\epsilon(t) &=
     \operatorname{sup} \left\{|z'| : z \in W, \phi\left(|z'|^2\right) \leq t^{\frac{1}{(1+\epsilon)}}\right\} = \sqrt{\phi^{-1}\left(t^{\frac{1}{(1 + \epsilon)}}\right)}, \text{ and}\\
    d_2^\epsilon(t) &=
    \operatorname{sup}\left\{|z'| : z \in W, \phi\left(|z'|^2\right) \leq t^{\frac{1}{(1+\epsilon)^2}}\right\} = \sqrt{\phi^{-1}\left(t^{\frac{1}{(1 + \epsilon)^2}}\right)}.
\end{align*}
Define the scaling map $\Sigma : \mathbb{C}^{n+1} \to \mathbb{C}^{n+1}$
by
\begin{align}
    \Sigma(z_1, z') = \left(\frac{z_1}{t}, \frac{z'}{d^*(t)}\right),
\end{align}
where $z' = (z_2, \dots, z_{n+1})$.

Throughout this section, we assume that \( D \), \( W\) and $D_t^{\epsilon}$ are as described above.
\begin{Remark}
By the definitions of \( d^{*}, d_1^{\epsilon}, \) and \( d_2^{\epsilon} \), we have
    \begin{align}
        d^{*}(t) \leq d_1^{\epsilon}(t) \leq d_2^{\epsilon}(t),
    \end{align}
    for all sufficiently small \( \epsilon, t > 0 \), and using \eqref{asy of psi}, we obtain
    \begin{align}
        \lim_{t \to 0^{+}} \frac{d_1^{\epsilon}(t)}{d^{*}(t)} = (1 + \epsilon)^{\frac{1}{2m}},  \quad \text{and} \quad
        \lim_{t \to 0^{+}} \frac{d_2^{\epsilon}(t)}{d^{*}(t)} = (1 + \epsilon)^{\frac{1}{m}} \label{asy of d_2^{epsilon}}.        
    \end{align}
\end{Remark}
In the following proposition, we prove that, for every sufficiently small \(t>0\), there exists a negative plurisubharmonic function with a pole at \(q(t)=(-t,0)\) on a domain with an exponentially flat boundary point. This function will be used to obtain the H\"ormander \(L^2\)-estimate \eqref{16}, which plays a crucial role in the proof of the localization Lemma~\ref{localisation lemma 4}.
\begin{prop}\label{plur fun.}
    Let $D \subset \mathbb{C}^{n + 1}$ be a domain, possibly unbounded, and let $0 \in bD$. Assume $0$ is an exponentially flat boundary point and $q(t) = (-t, 0)$, then for $\epsilon > 0$, there exist a constant $t_0(\epsilon) > 0$ and a negative plurisubharmonic function $\psi_t^{\epsilon}$ on $D$ with a pole at $q(t)$ for each $t \in (0, t_0(\epsilon))$. 
\end{prop}
\begin{proof}
Choose a cut-off function $\widetilde{\chi} \in C_c^\infty(-1, 1)$ such that 
\begin{align*}
    \widetilde{\chi} = 1 \text{ on } (-1/2, 1/2) \quad \text{and} \quad 0 \leq \widetilde{\chi} \leq 1. 
\end{align*}
Let $\epsilon$ and $t$ be sufficiently small positive real numbers. Define
\begin{align}\label{f_t}
    f_t^{\epsilon}(z) = \widetilde{\chi}\left(\frac{|z - q(t)|^2}{t^{\frac{\epsilon}{4(1 + \epsilon)^2}}}\right)\operatorname{log}(|z - q(t)|^2).
\end{align}
Here, $f_t^{\epsilon}$ is a  
plurisubharmonic function on $\mathbb{B}^{n+1}(q(t), {t^{\frac{\epsilon}{8(1+\epsilon)^2}}}/{\sqrt{2}})
\cup
(\mathbb{B}^{n+1}(q(t), t^{\frac{\epsilon}{8(1+\epsilon)^2}}))^c =: A$, and $f_t^{\epsilon} \in C^{\infty}(\mathbb{C}^{n + 1}\setminus \{q(t)\})$.
Let $z \in A^c$ and $\xi \in \mathbb{C}^{n + 1}$, consider
\begin{align*}
    \sum_{i, j = 1}^{n + 1}\frac{\partial^2 f_t^{\epsilon}}{\partial z_i \partial \bar{z}_j}(z) \xi_i  \bar{\xi}_j &\geq   \sum_{i, j = 1}^{n + 1}\frac{\partial^2}{\partial z_i \partial \bar{z}_j} \widetilde{\chi}\left(\frac{|z - q(t)|^2 }{t^{\frac{\epsilon}{4 (1 + \epsilon)^2}}}\right) \operatorname{log}(| z - q(t) |^2) \xi_i \bar{\xi}_j \\
    &+ \sum_{i, j = 1}^{n + 1} \frac{\partial}{\partial z_i} \widetilde{\chi}\left(\frac{| z - q(t)|^2 }{t^{\frac{\epsilon}{4 (1 + \epsilon)^2}}}\right) \frac{\partial}{\partial \bar{z}_j} \operatorname{log}(| z - q(t)| ^2)\xi_i \bar{\xi}_j \\
    &+ \sum_{i, j = 1}^{n+ 1}\frac{\partial}{\partial \bar{z}_j} \widetilde{\chi}\left(\frac{|z - q(t)|^2 }{t^{\frac{\epsilon}{4 (1 + \epsilon)^2}}}\right) \frac{\partial}{\partial z_i} \operatorname{log}(|z - q(t)| ^2)\xi_i\bar{\xi}_j\\
    &\geq \frac{M|\xi|^2}{t^{\frac{\epsilon}{2(1 + \epsilon)^2}}}\operatorname{log}(|z - q(t)|^2) + \frac{2\big|\sum_{i = 1}^{n + 1}({z_i} - {q}_i(t)) \bar{\xi}_i\big|^2}{ {t^{\frac{\epsilon}{4(1 + \epsilon)^2}}|z -q(t)|^2}}\widetilde{\chi}'\left(\frac{|z - q(t)|^2}{ {t^{\frac{\epsilon}{4(1 + \epsilon)^2}}}}\right)\\
    &\geq \frac{M|\xi|^2}{{t^{\frac{\epsilon}{2(1 + \epsilon)^2}}}} \operatorname{log}({t}^{\frac{\epsilon}{4(1 + \epsilon)^2}}/2) - \frac{M|\xi|^2}{{t^{\frac{\epsilon}{4(1 + \epsilon)^2}}}}\\
    &\geq \frac{M \epsilon |\xi|^2}{{4(1 + \epsilon)^2 t^{\frac{\epsilon}{2(1 + \epsilon)^2}}}}\operatorname{log}t -  \frac{M|\xi|^2}{{t^{\frac{\epsilon}{2(1 + \epsilon)^2}}}} \operatorname{log}(2) - \frac{M|\xi|^2}{t^{\frac{\epsilon}{4(1 + \epsilon)^2}}}\\
    &\geq - \frac{M|\xi|^2}{t^{\frac{\epsilon}{2(1 + \epsilon)^2}}} \left( 1 - \frac{\epsilon}{4(1 + \epsilon)^2} \operatorname{log}t\right) \\
    &\geq - \frac{M |\xi|^2}{t^{\frac{\epsilon}{2(1 + \epsilon)^2}}} \left(1 + \frac{1}{t^{\frac{\epsilon}{4(1 + \epsilon)^2}}}\right)\\
    &\geq -\frac{M |\xi|^2}{t^{\frac{3\epsilon}{4(1 + \epsilon)^2}}},
\end{align*}
for $t \in (0, t_0(\epsilon))$, where $t_0(\epsilon) > 0$ is sufficiently small, and
constant \(M>0\) appearing above may vary from step to step and is independent of $\epsilon > 0$ and $t \in (0, t_0(\epsilon))$. Hence
\begin{align}
    \sum_{i, j = 1}^{n + 1} \frac{\partial^2 f_t^{\epsilon}}{\partial z_i \partial \bar{z}_j}(z)\xi_i \bar{\xi}_j \geq - \frac{M|\xi|^2}{t^{\frac{3\epsilon}{4(1 + \epsilon)^2}}},
\end{align}
for $\xi \in \mathbb{C}^{n + 1}$ and $z \in \mathbb{C}^{n + 1} \setminus \{q(t)\}$. Define 
\begin{align}\label{psi_t}
    \psi_t^{\epsilon} := \frac{1}{2}f_t^{\epsilon} + \frac{M}{t^{\frac{3 \epsilon}{4(1 + \epsilon)^2}}} \bigg(\theta - \sup_{D} {\theta} - 1\bigg),
\end{align}
where $\theta$ is the bounded plurisubharmonic function on $D$, defined in \eqref{40}. Therefore, it is easy to see that $\psi_t^{\epsilon}$ is a negative plurisubharmonic function on $D$ with pole at $q(t)$ for each $t \in \big(0, t_0(\epsilon)\big)$.
\end{proof}
We now prove the localization Lemma \ref{localisation lemma 4} for extremal functions. The proof is based on the techniques developed by Nikolov \cite{Nikolai}, Krantz--Yu \cite{Krantz-Yu 1996}, James \cite{James}, Huang--James--Li \cite{HJL25}, and the author's previous works \cites{Ravi_Bulletin, Ravi_Synergies}.
\begin{lem}\label{localisation lemma 4}
Let $D \subset \mathbb{C}^{n + 1}$ be a pseudoconvex domain, possibly unbounded, and let $0 \in bD$. Assume $0$ is an exponentially flat boundary point and $q(t) = (-t, 0)$, then for $\epsilon > 0$, there exists $t_0(\epsilon) > 0$ such that $q(t) \in D^{*}$ for each $t \in (0, t_0(\epsilon))$, and 
\begin{align*}
\lim_{t \to 0^+} \frac{I_D^{0}(q(t))}{I_{D_t^\epsilon}^{0}(q(t))} &= 1, \qquad \quad \quad
\lim_{t \to 0^+} \frac{I_D^{1}(q(t);\xi(t))}{I_{D_t^\epsilon}^{1}(q(t);\xi(t))} = 1,\\
\lim_{t \to 0^+} \frac{\lambda_D^{k}(q(t))}{\lambda_{D_t^\epsilon}^{k}(q(t))} &= 1,
\qquad \quad \quad
\lim_{t \to 0^+} \frac{I_D(q(t);\xi(t))}{I_{D_t^\epsilon}(q(t);\xi(t))} = 1,
\end{align*}
where $\xi(t) \in \mathbb{C}^{n + 1} \setminus \{0\}$, and $k \in \{1, \dots, n+1\}$.
\end{lem}
\begin{proof}
    Let $\epsilon > 0$. Choose a cut-off function $\chi \in C_c^\infty(\mathbb{D} \times \mathbb{B}^n(0, 2))$ such that 
\begin{equation*}
    0 \leq \chi \leq 1, \text{ and} \quad \chi = 1 \text{ on } \mathbb{B}^1(0, 1/2) \times \mathbb{B}^{n}. 
\end{equation*}
Define,
\begin{equation*}
    \chi_t(z) = \chi\left(\frac{z_1}{t^{\frac{1}{(1+\epsilon)^2}}}, \frac{z'}{d_1^{\epsilon}(t)}\right) \text{ for } z \in (z_1, z') \in \mathbb{C}^{n + 1}.
\end{equation*}
Let $ z \in \widetilde{D}_t^\epsilon := \left\{z \in D_t^{\epsilon}: \left|h_{\epsilon}(z)\right| > \operatorname{exp}\left(- c_0^{\epsilon}t^{{1}/{(1+\epsilon)^2}}\right)\right\}$, where $c_0^{\epsilon} = \operatorname{cos}\left(\frac{\pi}{2(1 + \epsilon)}\right)$ and $h_{\epsilon}$ is defined in \eqref{hepsilon}.
Then, 
\[|z_1| \leq t^{\frac{1}{(1+\epsilon)}}\] 
and there exists $t_0(\epsilon) > 0$ such that
\begin{align*}
    \frac{|z_1|}{t^{\frac{1}{(1+\epsilon)^2}}} \leq t^\frac{\epsilon}{(1+\epsilon)^2} < 1/2, \text{  and  }
    \frac{|z'|}{d_1^\epsilon(t)} < 1,
\end{align*}
for $t \in (0, t_0(\epsilon))$. Therefore, $\chi_t = 1$ on $\widetilde{D}_t^\epsilon$. Let
\begin{equation*}
V_t = W \cap \left\{z \in \mathbb{C}^{n+1} : |\operatorname{Re}z_1| < t^\frac{1}{(1+\epsilon)^2} \right\},
\end{equation*}
where $W$ is defined at the beginning of this Section.
Note that $\chi_t \in C_c^\infty(V_t)$, and $V_t \cap D = D_t^\epsilon$. Let $f \in A^2(D_t^\epsilon)$ and $k \in \mathbb{N}$. Making $t_0(\epsilon) > 0$ smaller, if necessary, set
\begin{align}
    \alpha &= \overline{\partial}\left(f\chi_t h_{\epsilon}^k\right) \text{ on } D, \text{ and}\\
    \phi_t^{\epsilon} &= (2n+8) \psi_t^{\epsilon} + (\theta - \sup_{D}\theta - 1)\label{phi_t}, 
\end{align}
where $\psi_t^{\epsilon}$ is defined in \eqref{psi_t}, for $t \in (0, t_0(\epsilon))$. Since $\phi_t^{\epsilon} - |z|^2$ is plurisubharmonic on $D_t^{\epsilon}$, as $\theta(z) = |z|^2$ on $D_t^{\epsilon}$, and $\alpha$ is a $C^{\infty}$-smooth $(0, 1)$-form on $D$ satisfying $\bar{\partial}\alpha = 0$ and $\operatorname{supp} \alpha \subset D_t^{\epsilon}$, using Theorem \ref{Prop L2}, there exists $u \in C^{\infty}(D)$ satisfying
\begin{equation*}
    \overline{\partial}u = \alpha \text{ on } D,
\end{equation*}and
\begin{equation}\label{Hor inequality 2}
\int_{D} {|u|^2e^{-\phi_t^{\epsilon}}}\, \mathrm{d}V \leq \int_{D}|\alpha|^2 e^{-\phi_t^{\epsilon}} \, \mathrm{d}V.
\end{equation}
Now,
\begin{align}\label{3.887}
   \int_{D}|\alpha|^2 e^{-\phi_t^{\epsilon}} \, \mathrm{d}V &= \int_{D\cap V_t}{|f|^2|\overline{\partial}{\chi_t}|^2|h_{\epsilon}|^{2k}} e^{-\phi_t^{\epsilon}} \, \mathrm{d}V \leq \frac{Ca_t^{2k}}{t^2}\int_{D_t^\epsilon \setminus 
    \widetilde{D}_t^\epsilon}{|f|^2} e^{-\phi_t^{\epsilon}}\, \mathrm{d}V,
\end{align}
where $a_t = \operatorname{exp}\left(-c_0^{\epsilon}t^{{1}/{(1+\epsilon)^2}}\right)$, and $C > 0$ depends on
first-order derivatives of $\chi$. 

Let \(z \in D_t^{\epsilon} \setminus \widetilde{D}_t^{\epsilon}\). By \eqref{f_t}, \eqref{psi_t} and \eqref{phi_t}, after possibly choosing \(t_0(\epsilon)>0\) smaller, we obtain
\begin{align}
    \phi_t^{\epsilon}(z) &= (2n + 8) \widetilde{\chi}\left(\frac{|z - q(t)|^2}{t^{\frac{\epsilon}{4(1 + \epsilon)^2}}}\right) \operatorname{log}(|z - q(t)|) + \left(\frac{(2n + 8)M}{t^{\frac{3 \epsilon}{4(1 + \epsilon)^2}}} + 1\right)\left(\theta - \sup_D\theta - 1\right)\label{30}\\
    &\geq (2n + 8)\operatorname{log}(|z - q(t)|) - {t^{-\frac{7 \epsilon}{8(1 + \epsilon)^2}}} + \left(\inf_D \theta - \sup_D \theta - 1\right)\label{est of phi},
\end{align}
for each $t \in (0, t_0(\epsilon))$.

Using \eqref{3.887} and \eqref{est of phi}, we conclude
\begin{align}\label{21}
    \int_D|\alpha|^2 e^{-\phi_t^{\epsilon}} \, \mathrm{d}V \leq \frac{C_1a_t^{2k}}{t^2}\int_{D_t^{\epsilon} \setminus \widetilde{D}_t^{\epsilon}}\frac{|f|^2}{|z - q(t)|^{2n + 8}}e^{t^{-\frac{7 \epsilon}{8(1 + \epsilon)^2}}} \, \mathrm{d}V.
\end{align}

Making $t_0(\epsilon) > 0$ smaller, if necessary, we have
\begin{align}\label{3.89}
    |z - q(t)| \geq |z_1 + t| \geq t,
\end{align}
for each $z \in D_t^\epsilon \setminus 
    \widetilde{D}_t^\epsilon$ and $t \in (0, t_0(\epsilon))$. Using \eqref{21} and \eqref{3.89}, we get 
\begin{align}\label{estimate alpha 2}
    \int_{D}|\alpha|^2 e^{-\phi_t^{\epsilon}} \, \mathrm{d}V &\leq \frac{C_1 a_t^{2k}\operatorname{exp}\left(t^{-\frac{7\epsilon}{8(1 + \epsilon)^2}}\right)}{t^{(2n + 10)}} \norm{f}^2_{L^2(D_t^{\epsilon})},
\end{align}
for each $t \in (0, t_0(\epsilon))$.

From \eqref{Hor inequality 2} and
\eqref{estimate alpha 2}, we have
\begin{equation}\label{res equation}
    \norm{u}_{L^2(D)}^2 \leq \int_D |u|^2 e^{-\phi_t^{\epsilon}} \, \mathrm{d}V \leq \frac{C_1 a_t^{2k} \operatorname{exp}\left(t^{-\frac{7 \epsilon}{8(1 + \epsilon)^2}}\right)}{t^{(2n + 10)}}\norm{f}_{L^2(D_t^\epsilon)}^2. 
\end{equation}
Making $t_0(\epsilon) > 0$ smaller, if necessary, we have $\mathbb{B}^{n + 1}(q(t), t) \subset D$ and from \eqref{30}, we conclude
\[\phi_t^{\epsilon}(z) \leq (2n + 8)\operatorname{log}(|z - q(t)|),\]
for $z \in \mathbb{B}^{n + 1}(q(t), t)$ and $t \in (0, t_0(\epsilon))$.
Consider
\begin{equation}\label{19}
   \infty > \int_D {|u|^2e^{-\phi_t^{\epsilon}}} \, \mathrm{d}V \geq \int_{\mathbb{B}^{n + 1}(q(t), t)} {|u|^2e^{-\phi_t^{\epsilon}}} \, \mathrm{d}V \geq \int_{\mathbb{B}^{n + 1}(q(t), t)}  \frac{|u|^2}{|z - q(t)|^{(2n+8)}} \, \mathrm{d}V.
\end{equation}
From the above equation \eqref{19}, we conclude
\begin{equation}
        \frac{\partial^{|\alpha| + |\beta|} u}{\partial z^{\alpha} \partial \bar{z}^{\beta}}(q(t)) = 0 \text{ for all multi-indices $\alpha, \beta$ with $|\alpha| + |\beta| \leq 2$.}
\end{equation}
\textbf{Step 1.} We now apply the above for the $f \in A^2\left(D_t^{\epsilon}\right)$ with $\norm{f}_{L^2(D_t^{\epsilon})} \leq 1$ that satisfies 
\begin{equation*}
    I_{D_t^{\epsilon}}^0(q(t)) = |f(q(t))|^2 > 0.
\end{equation*}
Here, $I_{D_t^{\epsilon}}^0(q(t)) > 0$ as $D_t^{\epsilon}$ is bounded.
Set
\begin{equation*}
    g = \frac{\chi_t f h_{\epsilon}^k - u}{\left(h_{\epsilon}(q(t)\right)^k} \text{ on $D$}.
\end{equation*}
It follows that $g$
is holomorphic on $D$, $g(q(t)) = f(q(t))$, and using \eqref{res equation}, we obtain
\begin{align}
    \norm{g}_{L^2(D)} &\leq \frac{\norm{f}_{L^2(D_t^{\epsilon})} + \norm{u}_{L^2(D_t)}}{\left(h_{\epsilon}(q(t))\right)^k} 
    \leq \frac{1 + \frac{\sqrt{C_1} a_t^{k}\operatorname{exp}\left(t^{-\frac{7\epsilon}{8(1 + \epsilon)^2}}/2\right)}{t^{(n+5)}}}{\left(h_{\epsilon}(q(t))\right)^k} < \infty \label{L2 norm of g}.
\end{align}
Since $g(q(t)) \neq 0$, $A^2(D)$ is base-point free at $q(t)$.
Consider
\begin{align}
    \frac{I_{D}^0(q(t))}{I_{D_t^{\epsilon}}^0(q(t))} &\geq \frac{|g(q(t))|^2}{\norm{g}_{L^2(D)}^2 \cdot |f(q(t))|^2} \geq \left(\frac{(h_{\epsilon}(q(t)))^k}{1 + \frac{\sqrt{C_1} a_t^{k}\operatorname{exp}\left(t^{-\frac{7\epsilon}{8(1 + \epsilon)^2}}/2\right)}{t^{(n+5)}}}\right)^2 \label{3.88},
\end{align}
for each $t \in (0, t_0(\epsilon))$. For every $\epsilon > 0$, there exist $c_{\epsilon} > 0$, such that
\begin{align}\label{est of cepsilon}
    \frac{8 + 7\epsilon}{8(1 + \epsilon)^2} < c_{\epsilon} < \frac{1}{(1+\epsilon)}.
\end{align}
Choosing $k$ to be the greatest integer of $t^{-c_{\epsilon}}$. Hence
\begin{align}\label{est of k}
   t^{-c_{\epsilon}} - 1  < k \leq t^{-c_{\epsilon}}.
\end{align}
Using \eqref{3.88} and \eqref{est of k}, we have
\begin{align*}
     \frac{I_{D}^0(q(t))}{I_{D_t^\epsilon}^0(q(t))} &\geq \left(\frac{\left(h_{\epsilon}(q(t)\right)^k}{1 + \frac{\sqrt{C_1} a_t^{k}\operatorname{exp}\left(t^{-\frac{7\epsilon}{8(1 + \epsilon)^2}}/2\right)}{t^{(n+5)}}}\right)^2\\
     &\geq \left(\frac{\operatorname{exp}(-t^{\frac{1}{1 + \epsilon} - c_{\epsilon}})}{1 + \sqrt{C_1}\frac{\operatorname{exp}\left(-c_0^{\epsilon}(t^{-c_{\epsilon}} - 1)t^{\frac{1}{(1 + \epsilon)^2}} + t^{-\frac{7\epsilon}{8(1 + \epsilon)^2}}/2\right)}{t^{(n + 5)}}}\right)^2\\
     &\geq \left(\frac{\operatorname{exp}(-t^{\frac{1}{1 + \epsilon} - c_{\epsilon}})}{1 + \sqrt{C_1}\frac{\operatorname{exp}\left(\frac{-c_0^{\epsilon} + \left(c_0^{\epsilon}-t^{\frac{8 + 7\epsilon}{8(1 + \epsilon)^2}} + 1/2\right)t^{c_{\epsilon} - \frac{8 + 7\epsilon}{8(1 + \epsilon)^2}}}{t^{\frac{7\epsilon}{8(1 + \epsilon)^2} + c_{\epsilon} - \frac{8 + 7\epsilon}{8(1 + \epsilon)^2}}}\right)}{t^{(n + 5)}}}\right)^2
\end{align*}
Using \eqref{est of cepsilon}, we conclude
\begin{align}\label{lim sup of I_0}
    \liminf_{t \to 0^{+}} \frac{I_D^0(q(t))}{I_{D_t^{\epsilon}}^0(q(t))} \geq 1.
\end{align}
Since $D_t^{\epsilon} \subset D$ and $I^0_D$ is monotone decreasing with respect to $D$,
\begin{align}\label{lim inf I_0}
    \limsup_{t \to 0^{+}}\frac{I_D^0(q(t))}{I_{D_t^{\epsilon}}^0(q(t))} \leq 1.
\end{align}
Combining \eqref{lim sup of I_0} and \eqref{lim inf I_0}, we obtain
\begin{align}\label{loc of I0}
\lim_{t \to 0^+} {\frac{I_{D}^0(q(t))}{I_{D_t^\epsilon}^0(q(t))}} = 1,
\end{align}
for each $\epsilon > 0$.

Similarly, $A^2(D)$ is separates holomorphic directions at $q(t)$, for each $t \in (0, t_0(\epsilon))$, and for $\xi(t) \in \mathbb{C}^{n + 1}\setminus \{0\}$, we have
\begin{align}
    \lim_{t \to 0^+} {\frac{I_{D}^1\left(q(t); \xi(t) \right)}{I_{D_t^\epsilon}^{1}\left(q(t); \xi(t) \right)}} = 1, \quad \text{and} \quad \label{lim of I1}
    \lim_{t \to 0^+} \frac{\lambda_D^{k}(q(t))}{\lambda_{D_t^\epsilon}^{k}(q(t))} = 1, \quad \forall \, k \in \{1, \dots, n + 1\},
\end{align}
for each $\epsilon > 0$.

\textbf{Step 2.}
Since the extremal function $I_{D}$ and $\xi \overline{G}_D^{-1} \xi^{*}$, for $\xi \in \mathbb{C}^{n + 1} \setminus \{0\}$, may not be monotone with respect to domains, we have utilized Proposition \ref{est. of G} to prove localization of the extremal function $I_D$.

Let $\tilde{f} \in A^2(D_t^{\epsilon})$ be the maximum function for $I_{D_t^{\epsilon}}(q(t); \xi(t))$, i.e., 
\begin{align*}
   I_{D_t^{\epsilon}}(q(t); \xi(t)) = \xi(t)\tilde{f}''(q(t))\overline{G}^{-1}_{D
_t^{\epsilon}}(q(t)) \overline{\tilde{f}''(q(t))}\xi(t)^{*},
\end{align*}
$\tilde{f}(q(t)) = 0, \, \tilde{f}'(q(t)) = 0$, and 
${\lVert \tilde{f} 
\bigr\rVert_{L^2(D_t^{\epsilon})}} \leq 1$.

Define 
\[\tilde{g} = \frac{\chi_t \tilde{f} h_{\epsilon}^k - u}{\left(h_{\epsilon}(q(t)\right)^k} \text{ on $D$}.\]
It follows that $\tilde{g}$ is holomorphic on $D$, $\tilde{g}(q(t)) = 0, \, \tilde{g}'(q(t)) = 0, \, \tilde{g}''(q(t)) = \tilde{f}''(q(t))$, and using \eqref{res equation}, we obtain
\begin{align}\label{L2 norm of gtilde}
    \norm{\tilde{g}}_{L^2(D)} &\leq \frac{{\lVert \tilde{f} \bigr\rVert}_{L^2(D_t^{\epsilon})} + \norm{u}_{L^2(D_t)}}{\left(h_{\epsilon}(q(t))\right)^k} 
    \leq \frac{1 + \frac{\sqrt{C_1} a_t^{k}\operatorname{exp}\left(t^{-\frac{7\epsilon}{8(1 + \epsilon)^2}}/2\right)}{t^{(n+5)}}}{\left(h_{\epsilon}(q(t))\right)^k} < \infty.
\end{align}  

Consider
\begin{align*}
        I_{D}(q(t); \xi(t)) &\geq \norm{\tilde{g}}_{L^2(D)}^{-2}\xi(t)\tilde{g}''(q(t))\overline{G}^{-1}_{D}(q(t))\overline{\tilde{g}''(q(t))}\xi(t)^{*}\\
        &\geq \norm{\tilde{g}}_{L^2(D)}^{-2} \frac{\kappa_{D}(q(t))}{\kappa_{D_t^{\epsilon}}(q(t))}I_{D^{\epsilon}_t}((q(t)); \xi(t)). \qquad (\text{by applying Proposition \ref{est. of G}})
\end{align*}
Using \eqref{L2 norm of gtilde}, we obtain
\begin{align*}
     \frac{ I_{D}(q(t); \xi(t))}{I_{D^{\epsilon}_t}(q(t); \xi(t))} \geq \left(\frac{(h_{\epsilon}(q(t)))^k}{1 + \frac{\sqrt{C_1} a_t^{k}\operatorname{exp}\left(t^{-\frac{7\epsilon}{8(1 + \epsilon)^2}}/2\right)}{t^{(n+5)}}}\right)^2 \frac{\kappa_{D}(q(t))}{\kappa_{D_t^{\epsilon}}(q(t))},
\end{align*}
for $t \in (0, t_0(\epsilon))$. Choosing $k \in \mathbb{N}$ as in \eqref{est of k}, and using Proposition \ref{Fuchs} together with \eqref{loc of I0}, we conclude
\begin{align}\label{96}
\liminf_{t \to 0^+} {\frac{I_{D}(q(t); \xi(t))}{I_{D_t^\epsilon}\left(q(t); \xi(t)\right)}} \geq 1.
\end{align}
Since $D_t^{\epsilon} \subset D$ and $q(t) \in D^{*}$, Proposition \ref{Pro 2.1} and Proposition \ref{est. of G} imply that
\begin{align*}
    \frac{I_{D}(q(t); \xi(t))}{I_{D_t^\epsilon}\left(q(t); \xi(t)\right)} \leq \frac{\kappa^n_{D_t^{\epsilon}}(q(t))}{\kappa^n_{D}(q(t))} \frac{J_{D_t^{\epsilon}}(q(t))}{J_{D}(q(t))}.
\end{align*}
Using Proposition \ref{Fuchs}, Proposition \ref{Pro 2.1}, \eqref{loc of I0} and \eqref{lim of I1}, we conclude
\begin{align}\label{98}
    \limsup_{t \to 0^+} {\frac{I_{D}(q(t); \xi(t))}{I_{D_t^\epsilon}\left(q(t); \xi(t)\right)}}
    \leq 1,
\end{align}
Combining \eqref{96} and \eqref{98}, we obtain
\begin{align*}
    \lim_{t \to 0^+} {\frac{I_{D}(q(t); \xi(t))}{I_{D_t^\epsilon}\left(q(t); \xi(t)\right)}} = 1,
\end{align*}
for each $\epsilon > 0$.
\end{proof}
We next recall the scaling lemma from \cite{Ravi_Bulletin}, which states that the limiting domain of $f \circ \Sigma (D_t^\epsilon)$ is $\mathbb{D} \times \mathbb{B}^n$, where
$f: \{z \in \mathbb{C}^{n + 1}: \operatorname{Re}z_1 < 0\} \to \mathbb{D} \times \mathbb{C}^n$ is the biholomorphism given by
\begin{align}
    f(z_1, z') = \left(\frac{1+z_1}{1-z_1}, z'\right),
\end{align}
and the scaling map $\Sigma$ and $D_t^{\epsilon}$ are defined at the beginning of this section.
\begin{lem}[\cite{Ravi_Bulletin}*{Lemma $3.4$}]\label{ScalingLemma2}
    Let $D \subset \mathbb{C}^{n + 1}$ and $D_t^{\epsilon} \subset \mathbb{C}^{n + 1}$ be as defined at the beginning of Section \ref{section 3}. Then, for every $\epsilon, \delta > 0,$ there exists $t_0(\delta, \epsilon) > 0$ such that
    \begin{equation}\label{inclusions 2}
        (1-\delta)\left(\mathbb{D}\times \mathbb{B}^n\right) \subset f \circ \Sigma \left(D_t^\epsilon\right) \subset
        \mathbb{D} \times \mathbb{B}^n(0, {d}_2^\epsilon(t)/d^*(t)),
    \end{equation}
for each $ 0 < t < t_0(\delta, \epsilon
)$. 
\end{lem}
We now combine the localization Lemma \ref{localisation lemma 4} and the scaling lemma \ref{ScalingLemma2} to obtain the asymptotic behavior of extremal functions below. The proof follows the same argument as in \cite{Ravi_Bulletin}*{Lemma $3.6$} and \cite{Ravi_Synergies}*{Lemma $3.4$}.
\begin{thm}\label{I_1 2}
     Let $D \subset \mathbb{C}^{n + 1}$ be a pseudoconvex domain, possibly unbounded. Assume $0$ is an exponentially flat boundary point and $q(t) = (-t, 0)$, then, for $\xi \in \mathbb{C}^{n + 1} \setminus \{0\}$,
    
\begin{enumerate}[labelsep=10mm,leftmargin=20mm,itemsep=5mm]
    \item $\begin{aligned}[t] \lim_{t \to 0^+} \frac{(2t)^2(d^{*}(t))^{2n}I_{D}^{0}(q(t))}{I_{\mathbb{D} \times \mathbb{B}^{n}}^{0}(q(t))} = 1, 
\end{aligned}$
\item $\begin{aligned}[t] 
\lim_{t \to 0^+} \frac{(2t)^2(d^{*}(t))^{2n}I_D^{1}(q(t);(f \circ \Sigma)'(q(t)) \xi)}{I_{\mathbb{D} \times \mathbb{B}^{n}}^{1}(q(t);\xi)} = 1,
\end{aligned}$
\item $\begin{aligned}[t] 
\lim_{t \to 0^+} \frac{(2t)^{2(n + 2)}(d^{*}(t))^{2n(n + 2)}\lambda_{D}(q(t))}{\lambda_{\mathbb{D} \times \mathbb{B}^{n}}(q(t))} = 1, \text{ and}
\end{aligned}$
\item $\begin{aligned}[t]
\lim_{t \to 0^+} \frac{(2t)^{2(n + 2)}(d^{*}(t))^{2n(n + 2)}M_D(q(t);(f \circ \Sigma)'(q(t)) \xi)}{M_{\mathbb{D} \times \mathbb{B}^{n}}(q(t);\xi)} = 1.
\end{aligned}$
\end{enumerate}
\end{thm}
We now use the asymptotic behavior of extremal functions to establish the asymptotic behavior of the Bergman kernel, the Bergman metric, the Bergman canonical invariant, and the Ricci curvature at an exponentially flat infinite type boundary point.
\begin{proof}[Proof of Theorem \ref{1.3}] 
    Let $\xi \in \mathbb{C}^{n + 1} \setminus \{0\}$. By using Proposition \ref{Fuchs}, Propostion \ref{Pro 2.1} and Theorem \ref{I_1 2}, we obtain
\begin{align}
    &\lim_{t \to 0^+}
    \frac{\kappa_D(q(t))(2t)^2(d^{*}(t))^{2n}}
    {K_{\mathbb{D} \times B_n(0,1)}(0)}
    =1,\qquad
    \lim_{t \to 0^{+}}
    \frac{B_D(q(t);\xi)}
    {B_{\mathbb{D}\times B_n(0,1)}
    (0;(f\circ\Sigma)'(q(t))\xi)}
    =1, \\
    &\lim_{t \to 0^{+}}
    \frac{J_D(q(t))}
    {J_{\mathbb{D}\times B_n(0,1)}(0)}
    =1,\qquad
    \text{and}\qquad
    \lim_{t \to 0^{+}}
    \frac{(n+1)-R_D(q(t);\xi)}
    {(n+1)-R_{\mathbb{D}\times B_n(0,1)}
    (0;(f\circ\Sigma)'(q(t))\xi)}
    =1.
\end{align}
We get the theorem by using Lemma \ref{Formula for kernel}.
\end{proof}
\begin{Remark}
For $\alpha, N > 0$, let
\begin{equation}
    C_{\alpha,N} := \Big\{\big(z_1,z'\big) \in \mathbb{C}\times\mathbb{C}^{n} :
    \operatorname{Re}z_1 < -\alpha|z'|^N\Big\},
\end{equation}
where $z' = (z_2, \dots, z_{n+1})$.
An \emph{$(\alpha, N)$-cone type stream} approaching the exponentially flat boundary point $0$ of the domain $D$ is a continuous curve
$q:(0,\epsilon_0)\to D\cap C_{\alpha,N}$
such that
\[
\lim_{t\to0^+} q(t)=0,
\]
for some $\epsilon_0>0$.

The asymptotic behavior of the Bergman kernel, the Bergman metric, the Bergman canonical invariant, and the Ricci curvature also holds along any $(\alpha, N)$-cone type stream approaching the exponentially flat boundary point $0\in bD$. This follows by adapting the geometric constructions developed in \cite{Ravi_Bulletin} together with the methods of this section.
\end{Remark}
\section{Bergman--Einstein metrics on exponentially flat infinite type domains}\label{Section 4}
In this section, we use the asymptotic behavior of biholomorphic invariants to prove that the Bergman metric can not be Einstein on a (possibly unbounded) domain with an exponentially flat boundary point.

The following proposition shows that the Bergman metric is K\"ahler--Einstein if and only if the Bergman canonical invariant is constant. This result was proved in \cite{HJL25}*{Remark $3.9$}, we include a proof here for the reader's convenience.
\begin{prop}\label{Prop.}
    Let $D \subset \mathbb{C}^{n + 1}$ be a pseudoconvex domain, possibly unbounded, and let $0 \in bD$. Assume $0$ is an exponentially flat boundary point. Let $U$ be a neighbourhood of the origin such that $D \cap U$ is connected, and on which the Bergman metric $B_{D}$ is well defined. Then, the Bergman metric $B_{D}$ is K\"ahler--Einstein on $D \cap U$ if and only if its Bergman invariant $J_{D} \equiv (n + 2)^{n + 1} \pi^{n + 1}/(n + 1)!$ on $D \cap U$.
\end{prop}
\begin{proof}
    Since $0$ is an exponentially flat boundary point, there exists a neighbourhood $U_0 \subset U$ of the origin such that
\begin{align}
    D \cap U_0 = \{z \in U_0: \rho(z) < 0\},
\end{align}
where the defining function $\rho(z_1, z') = \operatorname{Re}z_1 + \phi \left(|z'|^2\right)$ for $(z_1, z') \in \mathbb{C}^{n + 1}$. Here, $\phi: \mathbb{R} \to \mathbb{R}$ is a $C^{\infty}$-smooth function that is exponentially flat at the origin. Hence, from the definition, there exists $\epsilon_0 > 0$ such that 
\begin{align}\label{phi''}
   \phi''(x) > 0 \, \, \text{for} \, \, 0 < x < \epsilon_0.
\end{align}
From \eqref{phi''}, $\phi'$ is strictly increasing on $(0, \epsilon_0)$. Hence $\phi'(x) > \phi'(0) = 0$ for $x \in (0, \epsilon_0)$.
Therefore, $bD \cap U_0 \cap \left\{z \in \mathbb{C}^{n + 1}: 0 < |z'| < \sqrt{\epsilon_0}\right\}$ is the strongly pseudoconvex part of the $bD$. By \cite{HJL25}*{Corollary $3.7$}, for any $p \in bD \cap U_0 \cap \{z \in \mathbb{C}^{n + 1}: 0 < |z'| < \sqrt{\epsilon_0}\}$, one has
\begin{align}\label{invariant on str point}
    \lim_{z \to p}J_{D}(z) = \frac{(n + 2)^{n + 1}\pi^{n + 1}}{(n + 1)!}.
\end{align}
Since $U_0$ is a neighbourhood of the origin, there exists $\epsilon > 0$ such that 
\[\mathbb{B}^1(0, \epsilon) \times \mathbb{B}^n\left(0, \sqrt{\phi^{-1}(\epsilon)}\right) \subset U_0 \quad \text{and} \quad \phi^{-1}(\epsilon) < \epsilon_0.\]
The Bergman metric $B_D$ is K\"ahler--Einstein if its Ricci curvature $R_{D} = cB_{D}$ for some constant $c$. By Theorem \ref{1.3}, one has $c = -1$. Consequently, the K\"ahler--Einstein assumption implies that $\operatorname{log}J_D$ is a pluriharmonic function on $D \cap U$.

Let $t \in \mathbb{B}^{1}(0, \epsilon)$ with 
$\operatorname{Re}(t) \in (-\epsilon, 0)$, then $J_D(t, .)$ is harmonic on $\{z' \in \mathbb{C}^n: 
|z'| < \sqrt{\phi^{-1}(-\operatorname{Re}(t))}\}$ as $\{(t, z') \in \mathbb{C}^{n + 1}: |z'| < \sqrt{\phi^{-1}(-\operatorname{Re}(t))}\} \subset D \cap U_0$. Since
\[\{(t, z') \in \mathbb{C}^{n + 1}: |z'| = \sqrt{\phi^{-1}(-\operatorname{Re}(t))}\} \subset bD \cap U_0 \cap \{z \in \mathbb{C}^{n + 1}: 0 < |z'| < \sqrt{\epsilon_0}\},\]
$J_D(t, .)$ is constant on $b\{z' \in \mathbb{C}^n: |z'| < \sqrt{\phi^{-1}(-\operatorname{Re}(t))}\}$ by \eqref{invariant on str point}. Hence, by the maximum principle, we conclude
\begin{align*}
    J_D(z) = \frac{(n + 2)^{n + 1}\pi^{n + 1}}{(n + 1)!},
\end{align*}
for $z \in \{(t, z') \in \mathbb{B}^1(0, \epsilon) \times \mathbb{B}^n(0,\sqrt{\phi^{-1}(\epsilon)}): \operatorname{Re}(t) \in (-\epsilon, 0), |z'| < \sqrt{\phi^{-1}(- \operatorname{Re}(t))}\} \subset D \cap U_0$. Since $J_D$ is real analytic on $D \cap U$ and $D \cap U$ is connected, 
\[J_D \equiv (n + 2)^{n + 1}\pi^{n + 1}/(n + 1)!.\]
Conversely, if $J_D$ takes a constant value on $D \cap U$, then it is easy to see that the Bergman metric is K\"ahler--Einstein on $D \cap U$.
\end{proof}
We now prove our main result.

\begin{proof}[Proof of Theorem \ref{Main theorem}]   
Here, \(D \subset \mathbb{C}^{n+1}\) is a pseudoconvex domain, and \(0 \in bD\) is a local holomorphic peak point (see \eqref{peak fun.}). By \cite{Nikolai 2002}*{Theorem $2$}, there exists a connected neighbourhood \(U\) of \(0\) such that \(D \cap U\) is connected and the Bergman metric is well defined on \(D \cap U\). In particular,
\[
D \cap U \subset D^{*}.
\]
    Assume the Bergman metric is Einstein on $D^{*}$. By Proposition \ref{Prop.}, we have
    \begin{align}
        J_D(z) = \frac{(n + 2)^{n + 1} \pi^{n + 1}}{(n + 1)!},
    \end{align}
    for $z \in D \cap U$. By Theorem \ref{1.3}, we conclude
    \begin{align}
        \frac{2 \pi^{n + 1}(n + 1)^{n}}{n!} = \frac{(n + 2)^{n + 1} \pi^{n + 1}}{(n + 1)!},
    \end{align}
which is not true for any $n \in \mathbb{N}$. Hence, the Bergman metric can not be Einstein on $D^{*}$.
\end{proof}

\section*{Acknowledgements}
I would like to thank Ilya Kossovskiy for drawing my attention to Cheng's conjecture.
I am partially supported by the NSFC Grant No. W$2431006$.
\def\MR#1{\relax\ifhmode\unskip\spacefactor3000 \space\fi%
  \href{http://www.ams.org/mathscinet-getitem?mr=#1}{MR#1}}
 
\end{document}